\documentclass[11pt]{amsart}
\usepackage[latin9]{inputenc}
\usepackage[a4paper]{geometry}
\usepackage{textcomp}
\usepackage{mathrsfs}
\usepackage{amstext}
\usepackage{amsthm}
\usepackage{amssymb}
\usepackage{wasysym}

\makeatletter

\newcommand{\lyxmathsym}[1]{\ifmmode\begingroup\def\b@ld{bold}
  \text{\ifx\math@version\b@ld\bfseries\fi#1}\endgroup\else#1\fi}

\numberwithin{equation}{section}
\numberwithin{figure}{section}

\usepackage[all]{xy}
\usepackage{indentfirst}
\usepackage{amscd}\usepackage{tikz}
\usepackage{tikz-cd}

\allowdisplaybreaks[2]

\newtheorem{theorem}{Theorem}[section]
\newtheorem{proposition}[theorem]{Proposition}\newtheorem{lemma}[theorem]{Lemma}\newtheorem{corollary}[theorem]{Corollary}\theoremstyle{definition}
\newtheorem{remark}[theorem]{Remark}\newtheorem{definition}[theorem]{Definition}\newtheorem{example}[theorem]{Example}\numberwithin{equation}{section}

\makeatother

\begin{document}
\title{Modular structure theory on Hom-Lie algebras}
\author{Dan Mao}
\address{D. Mao: School of Mathematics and Statistics, Northeast Normal University,
Changchun 130024, China}
\email{danmao678@nenu.edu.cn}
\author{Baoling Guan}
\address{B. Guan: College of Sciences, Qiqihar University, Qiqihar 161006,
China}
\email{baolingguan@126.com}
\author{Liangyun Chen$^{*}$}
\address{L. Chen: School of Mathematics and Statistics, Northeast Normal University,
Changchun 130024, China}
\email{chenly640@nenu.edu.cn}
\thanks{{*}Corresponding author.}
\thanks{ Supported by NSF of Jilin Province (No. YDZJ202201ZYTS589), NNSF
of China (Nos. 12271085, 12071405), the Fundamental Research Funds
for the Central Universities and the special project of basic business
for Heilongjiang Provincial Education Department (No. 145109133).}
\begin{abstract}
The aim of this paper is to transfer the restrictedness theory to
Hom-Lie algebras. The concept of restricted Hom-Lie algebras which
is introduced in \cite{BM2} will be used in this paper. First, the
existence of $p$-structures on a Hom-Lie algebra is studied and the
direct sum of restricted Hom-Lie algebras is analyzed. Then, the definition
of a restrictable Hom-Lie algebra is given and the equivalence relation
between restrictable Hom-Lie algebras and restricted Hom-Lie algebras
is constructed. Finally, the $p$-envelopes of a Hom-Lie algebra are
defined and studied.

\tableofcontents{}
\end{abstract}

\maketitle

\section{Introduction and preliminaries}

In physics, $q$-deformed Lie algebras naturally appeared when physicists
studied $q$-oscillators (see \cite{CKL}). In $q$-deformed Lie algebras,
the classical Jacobi identity was replaced by $q$-deformed Jacobi
identity. Then the class of algebras was formally defined as Hom-Lie
algebras in \cite{HLS,LS}, where the Hom-Jacobi identity was called
as Jacobi-like identity with the form

\[
\underset{x,y,z}{\leftturn}\left\langle (\varsigma+\mathrm{id})(x),\left\langle y,z\right\rangle _{\varsigma}\right\rangle _{\varsigma}=0.
\]

In \cite{LS}, Hom-Lie algebras were further generalized by twisting
not only the Jacobi identity, but also the skew-symmetry and the structure
map. Recently, many people attach importance to the study of Hom-Lie
algebras for its close relation to discrete and deformed vector fields
and differential calculus \cite{HLS,LS,LS2}. Particularly, the enveloping
algebras of Hom-Lie algebras were constructed and studied in \cite{Yau};
In \cite{MS}, the authors extended the 1-parameter formal deformation
theory to Hom-Lie algebras and constructed a cohomology theory adapted
to this deformation theory; Quadratic Hom-Lie algebras were introduced
and studied in \cite{BM}; All rigid complex 3-dimensional multiplicative
Hom-Lie algebras were obtained in \cite{AV} by studying all deformations
of multiplicative Hom-Lie algebras; Hom-Lie structures on tensor products
were studied and the representations of current Hom-Lie algebras were
considered in \cite{BMS}; In \cite{DS}, the authors investigated
the Nijenhuis operators on Hom-Lie algebras and defined the cohomology
associated to a Nijenhuis operator to study formal deformations on
a Hom-Lie algebra.

The concept of a restricted Lie algebra is attributable to Jacobson
(see \cite{Jac}). It plays a predominant role in modular Lie algebra
theory, especially in the classification of the finite-dimensional
simple modular Lie algebras \cite{SF,Str}. Because of its important
role in modular theory, restrictedness theory attracted great attention.
Bouarroudj, Lebedev, Leites and Shchepochkina studied the queerification
of restricted simple Lie algebras in characteristic 2 \cite{BLLS}.
Petrogradski and Wang generalized the concept to modular Lie superalgebras
in \cite{Pet} and \cite{WZ} respectively. Then Sun studied the restricted
envelopes of Lie superalgebras in \cite{SLGW}. Moreover, Guan \cite{GC}
and Bouarroudj \cite{BM2} defined restricted Hom-Lie algebras respectively.

Due to the importance of Hom-Lie algebras and restrictedness theory,
this paper attempts to transfer the restrictedness theory to Hom-Lie
algebras. Although Guan has studied the restrictedness theory of Hom-Lie
algebras in \cite{GC}, definition of restricted Hom-Lie algebras
therein is too limited to be appropriate to the consideration of some
problems (see \cite{BM2}). Therefore, the authors of \cite{BM2}
gave an alternative definition, which will be used in this paper.

The paper is organized as follows. In Section 1, after recalling and
giving several basic definitions, we point out the difference between
two definitions of restricted Hom-Lie algebras (Remark 1.9) and give
several examples (Examples 1.10, 1.11 and 1.12). In Section 2, we
give a sufficient condition of the existence of $p$-structures on
Hom-Lie algebras (Theorem 2.3). In addition, the direct sum of restricted
Hom-Lie algebras is studied. In particular, it is proved that the
direct sum of two restricted Hom-Lie algebras is still a restricted
Hom-Lie algebra (Proposition 2.4). As by-products, some properties
of the morphism of restricted Hom-Lie algebras are obtained (Thoerems
2.8 and 2.9). In Section 3, the definition of restrictable Hom-Lie
algebras is introduced and examples of which are shown (Examples 3.2,
3.3 and 3.4). Moreover, the equivalence relation between restricted
and restrictable Hom-Lie algebras is built (Theorem 3.5). In Section
4, we study the restricted envelopes of Hom-Lie algebras. We generalize
the notion of $p$-envelopes and its relatives to Hom-Lie algebras
and construct an elementary theory. As an application, the Hom-version
of Iwasawa's theorem is given (Theorem 4.8), which is a result on
the representation of a finite-dimensional Hom-Lie algebra.

Throughout this paper, $\mathbb{F}$ is the underlying field of prime
characteristic $p$ and $L$ is a Hom-Lie algebra over $\mathbb{F}$
unless otherwise specified. Next, we will recall some notions.

\begin{definition}(see \cite{She}) (1) A Hom-Lie algebra is a triple
$(L,[.,.]_{L},\alpha)$ consisting of a linear space \emph{L}, a skew-symmetric
bilinear map $[.,.]_{L}$ : $\Lambda{}^{2}$\emph{L} \textrightarrow{}
\emph{L} and a linear map $\alpha$ : $L$ \textrightarrow{} $L$
satisfying the following Hom-Jacobi identity:
\[
[\alpha(x),[y,z]_{L}]_{L}+[\alpha(y),[z,x]_{L}]_{L}+[\alpha(z),[x,y]_{L}]_{L}=0,
\]
for all $x,y,z\in L$;

(2) A Hom-Lie algebra is called a multiplicative Hom-Lie algebra if
$\alpha$ is an algebraic morphism, i.e., for any $x,y\in L$, we
have $\alpha([x,y]_{L})=[\alpha(x),\alpha(y)]_{L}$;

(3) A Hom-Lie algebra is called a regular Hom-Lie algebra if $\alpha$
is an algebraic automorphism;

(4) A sub-vector space $g\subset L$ is called a Hom-Lie subalgebra
of $(L,[.,.]_{L},\alpha)$ if $\alpha(g)\subset g$ and $g$ is closed
under the bracket operation $[.,.]_{L}$, i.e., $[x,y]_{L}\in g$
for all $x,y\in g$;

(5) A sub-vector space $g\subset L$ is called a Hom-Lie ideal of
$(L,[.,.]_{L},\alpha)$ if $\alpha(g)\subset g$ and $[x,y]_{L}\in g$
for all $x\in g,y\in L$.

\end{definition}

\begin{definition} (see \cite{BM}) A Hom-Lie algebra $(L,[.,.]_{L},\alpha)$
is called an involutive Hom-Lie algebra if it is multiplicative and
$\alpha$ is an involution, that is, $\alpha^{2}=\mathrm{id}$.

\end{definition}

\begin{definition} (see \cite{CH}) Let $(L,[.,.]_{L},\alpha)$ be
a Hom-Lie algebra. The sequence $(L^{n})_{n\in\mathbb{N}}$ of $L$
defined by $L^{1}:=L$, $L^{n+1}:=[L^{n},L]$ is called as the descending
central series of $L$. The Hom-Lie algebra $(L,[.,.]_{L},\alpha)$
is said to be nilpotent if there exists $k\in\mathbb{N}$ such that
$L^{k}=0$.

\end{definition}

\begin{definition} (see \cite{GC}) Let $(L,[.,.]_{L},\alpha)$ be
a multiplicative Hom-Lie algebra over $\mathbb{F}$ and $L_{1}=\{x\in L|\alpha(x)=x\}$.
A mapping $[p]:L_{1}\rightarrow L_{1}$, $x\mapsto x^{[p]}$ is called
a $p$-mapping, if

(1) $[\alpha(y),x^{[p]}]_{L}=(\mathrm{ad}x)^{p}(y),\forall x\in L_{1},y\in L$,

(2) $(kx)^{[p]}=k^{p}x^{[p]},\forall x\in L_{1},k\in\mathbb{F}$,

(3) $(x+y)^{[p]}=x^{[p]}+y^{[p]}+\sum_{i=1}^{p-1}s_{i}(x,y)$,

where $(\mathrm{ad}(x\otimes X+y\otimes1))^{p-1}(x\otimes1)=\sum_{i=1}^{p-1}is_{i}(x,y)\otimes X^{i-1}$
in $L\otimes_{\mathbb{F}}\mathbb{F}[X]$, $\forall x,y\in L_{1}$,
$\alpha(x\otimes X)=\alpha(x)\otimes X$. The quadruple $(L,[.,.]_{L},\alpha,[p])$
is referred to as a restricted Hom-Lie algebra.

\end{definition}

Since this definition is not appropriate to queerify a restricted
Hom-Lie algebras in characteristic $2$, an alternative definition
is given by \cite{BM2}, which is as follows:

\begin{definition}(see \cite{BM2}) Let $L$ be a multiplicative
Hom-Lie algebra in characteristic $p$ with a twist $\alpha$. A mapping
$[p]_{\alpha}:L\rightarrow L$ , $x\mapsto x^{[p]_{\alpha}}$ is called
a $p$-structure of $L$ and $L$ is said to be restricted if

(R1) $\mathrm{ad}(x^{[p]_{\alpha}})\lyxmathsym{\textopenbullet}\alpha^{p-1}=\mathrm{ad}(\alpha^{p-1}(x))\lyxmathsym{\textopenbullet}\mathrm{ad}(\alpha^{p-2}(x))\lyxmathsym{\textopenbullet}\cdot\cdot\cdot\lyxmathsym{\textopenbullet}\mathrm{ad}(x)$,
for all $x\in L$;

(R2) $(kx)^{[p]_{\alpha}}=k^{p}x^{[p]_{\alpha}}$, for all $x\in L$
and for all $k\in\mathbb{F}$;

(R3)$(x+y)^{[p]_{\alpha}}=x^{[p]_{\alpha}}+y^{[p]_{\alpha}}+\sum_{i=1}^{p-1}s_{i}(x,y)$,
where the $s_{i}(x,y)$ can be obtained from $\mathrm{ad}(\alpha^{p-2}(kx+y))\text{\textopenbullet}\mathrm{ad}(\alpha^{p-3}(kx+y))\text{\textopenbullet}...\text{\textopenbullet}\mathrm{ad}(kx+y)(x)=\sum_{i=1}^{p-1}is_{i}(a,b)k^{i-1}$.\end{definition}

Notice that in the case for $p=2$, the conditions (R1) and (R3) read, respectively, as
\[
[x^{[2]_\alpha},\alpha(y)]=[\alpha(x),[x,y]]\quad \textrm{and} \quad (x+y)^{[2]_\alpha}=x^{[2]_\alpha}+y^{[2]_\alpha}+[x,y].
\]

\begin{definition}Let $(L,[.,.]_{L},\alpha,[p]_{\alpha})$ be a restricted
Hom-Lie algebra. A Hom-Lie subalgebra (Hom-Lie ideal) $H$ of $L$
is said to be a $p$-subalgebra ($p$-ideal), if $x^{[p]_{\alpha}}\in H$,
for all $x\in H$.

\end{definition}

Let $S\subset L$ be a subset of a restricted Hom-Lie algebra $(L,[.,.]_{L},\alpha,[p]_{\alpha})$
and $S_{p}$ the intersection of all $p$-subalgebras containing $S$.
Clearly, $S_{p}$ is a $p$-subalgebra of $L$. It is usually called
as the $p$-subalgebra generated by $S$ in $L$.

\begin{definition}Let $(L,[.,.],\alpha)$ be a Hom-Lie algebra.

(1) A mapping $f:L\rightarrow L$ is called to be $p$-semilinear,
if $f(kx+y)=k^{p}f(x)+f(y)$, for all $x,y\in L,k\in\mathbb{F}$.

(2) Let $S$ be a subset of $L$, the $\alpha^{p-1}$-centralizer
of $S$ in $L$, which is denoted by $C_{L}(S)$, is defined as
\[
C_{L}(S)=\{x\in L|[x,\alpha^{p-1}(y)]=0,\forall y\in S\}.
\]

Particularly, if $S=L$, it is called the $\alpha^{p-1}$-center of
$L$ and it is denoted by $C(L)$.

\end{definition}

\begin{proposition} Let $(L,[.,.],\alpha)$ be a regular Hom-Lie
algebra and $[p]_{\alpha}$ a $p$-structure on $L$. Then its $\alpha^{p-1}$-center
$C(L)$ is a $p$-ideal of the regular restricted Hom-Lie algebra
$(L,[.,.],\alpha,[p]_{\alpha})$.

\end{proposition}
\begin{proof}
For any $x\in C(L)$, $y,z\in L$, we have
\[
[\alpha(x),\alpha^{p-1}(z)]=\alpha([x,\alpha^{p-1}(\alpha^{-1}(z))])=0
\]
and
\begin{align*}
[[x,y],\alpha^{p-1}(z)] & =[[y,\alpha^{p-2}(z)],\alpha(x)]+[[\alpha^{p-2}(z),x],\alpha(y)]\\
 & =[\alpha^{p}([\alpha^{-p}(y),\alpha^{-2}(z)]),\alpha(x)]+[[\alpha^{p-1}(\alpha^{-1}(z)),x],\alpha(y)]\\
 & =\alpha([\alpha^{p-1}([\alpha^{-p}(y),\alpha^{-2}(z)]),x])=0.
\end{align*}
It follows that $\alpha(x),[x,y]\in C(L)$. Therefore, $C(L)$ is
a Hom-Lie ideal of $L$.

Moreover, for any $x\in C(L)$, $y\in L$, it follows from
\[
[x^{[p]_{\alpha}},\alpha^{p-1}(y)]=[\alpha^{p-1}(x),[\cdots,[x,y]]]=[\alpha^{p-1}(x),[\cdots,[x,\alpha^{p-1}(\alpha^{1-p}(y))]]]=0
\]
that $x^{[p]_{\alpha}}\in C(L)$. Therefore, $C(L)$ is a $p$-ideal
of $L$.
\end{proof}
\begin{remark} In Definition $1.4$, the $p$-mapping of a restricted
Hom-Lie algebra is defined on $L_{1}$, where $L_{1}=\{x\in L|\alpha(x)=x\}$,
while the $p$-strcuture of a restricted Hom-Lie algebra in Definition
$1.5$ is defined on $L$. Furthermore, the conditions in Definition
$1.5$ are weaker than those in Definition $1.4$. In other words,
for a multiplicative Hom-Lie algebra $(L,[\cdot,\cdot]_{L},\alpha)$
, a given mapping $[p]:L\rightarrow L$ may make $L$ is restricted
under Definition $1.5$ but not under Definition $1.4$.

\end{remark}

This kind of examples will be provided here.

\begin{example} Let $(L^{-},[.,.]_{L})$ be a Lie algebra given by
an associative algebra $L$ via the commutator bracket $[.,.]_{L}$
and $f:L\rightarrow L,x\mapsto x^{p}$ the Frobenius mapping on $L$.
Then $(L^{-},[.,.]_{L},f)$ is a restricted Lie algebra (see \cite{SF}).
Given a Lie algebra morphism $\alpha$ on $L^{-}$, then $(L^{-},[.,.]_{\alpha},[p]_{\alpha},\alpha)$
is a restricted Hom-Lie algebra, where $[.,.]_{\alpha}=\alpha\text{\textopenbullet}[.,.]_{L}$
and $[p]_{\alpha}=\alpha^{p-1}\text{\textopenbullet}f$ (see \cite{BM2}).

\end{example}

\begin{example} For the two-dimensional nonabelian Lie algebra $L:=\mathbb{F}h\oplus\mathbb{F}x$
with the nontrivial bracket $[h,x]=x$, a mapping $[p]:L\rightarrow L$,
which is defined by means of $(\lambda h+\mu x)^{[p]}:=\lambda^{p}h+\lambda^{p-1}\mu x$,
$\forall\lambda,\mu\in\mathbb{F}$, is the unique $p$-mapping on
$L$ (see \cite{SF}). Let $\alpha:L\rightarrow L$ be a Lie algebra
morphism on $L$. Similarly, we obtain a restricted Hom-Lie algebra
$(L,[.,.]_{\alpha},\alpha,[p]_{\alpha})$ by setting $[.,.]_{\alpha}=\alpha\text{\textopenbullet}[.,.]$
and $x^{[p]_{\alpha}}=\alpha^{p-1}(x^{[p]})$ (see \cite{BM2}).

\end{example}

\begin{example} Let $(L,[.,.]_{L})$ be an abelian Lie algebra. Let
$\alpha:L\rightarrow L$ be a Lie algebra morphism and $[.,.]_{\alpha}=\alpha\text{\textopenbullet}[.,.]_{L}$
such that the triple $(L,[.,.]_{\alpha},\alpha)$ is an abelian multiplicative
Hom-Lie algebra. For a $p$-semilinear mapping $f:L\rightarrow L$,
let $f'=\alpha^{p-1}\circ f$. Then the quadruple $(L,[.,.]_{\alpha},\alpha,f')$
is a restricted Hom-Lie algebra.

\end{example}

Hereafter, the notion of restricted Hom-Lie algebras introduced in
\cite{BM2} (i.e. Definition 1.5) will be used.

\section{Properties of restricted Hom-Lie algebras}

Given a Hom-Lie algebra $L$, two questions naturally arise: Can we
define a $p$-structure on $L$? How many different $p$-structures
on $L$? The following lemma is given first as a preparation.

\begin{lemma}Let $(g,[.,.]_{g},\alpha,[p]_{\alpha})$ be a restricted
Hom-Lie algebra, $L\subset g$ a Hom-Lie subalgebra and $[p]_{\alpha}^{1}$
a mapping on $L$. Then the following statements are equivalent:

(1) $[p]_{\alpha}^{1}$ is a $p$-structure on $L$;

(2) There exists a $p$-semilinear mapping $f:L\rightarrow C_{g}(L)$
such that $[p]_{\alpha}^{1}=[p]_{\alpha}+f$.

\end{lemma}
\begin{proof}
(1)$\Rightarrow$(2) Consider $f:L\rightarrow g,f(x)=x^{[p]_{\alpha}^{1}}-x^{[p]_{\alpha}}$.
Since

\begin{align*}
 & \left[f(x),\alpha^{p-1}(y)\right]\\
= & \left[x^{[p]_{\alpha}^{1}},\alpha^{p-1}(y)\right]-\left[x^{[p]_{\alpha}},\alpha^{p-1}(y)\right]\\
= & \mathrm{ad}(x^{[p]_{\alpha}^{1}})\text{\textopenbullet}\alpha^{p-1}(y)-\mathrm{ad}(x^{[p]_{\alpha}})\text{\textopenbullet}\alpha^{p-1}(y)\\
= & \mathrm{ad}(\alpha^{p-1}(x))\text{\textopenbullet}\mathrm{ad}(\alpha^{p-2}(x))\text{\textopenbullet}\cdot\cdot\cdot\text{\textopenbullet}\mathrm{ad}(x)(y)\\
 & -\mathrm{ad}(\alpha^{p-1}(x))\text{\textopenbullet}\mathrm{ad}(\alpha^{p-2}(x))\text{\textopenbullet}\cdot\cdot\cdot\text{\textopenbullet}\mathrm{ad}(x)(y)\\
= & 0,
\end{align*}

for all $x,y\in L$, $f$ actually maps $L$ into $C_{g}(L)$. For
$x,y\in L,k\in\mathbb{F}$, we obtain

\begin{align*}
f(kx+y) & =(kx+y)^{[p]_{\alpha}^{1}}-(kx+y)^{[p]_{\alpha}}\\
 & =k^{p}x^{[p]_{\alpha}^{1}}+y^{[p]_{\alpha}^{1}}+\sum_{i=1}^{p-1}s_{i}(kx,y)-k^{p}x^{[p]_{\alpha}}-y^{[p]_{\alpha}}-\sum_{i=1}^{p-1}s_{i}(kx,y)\\
 & =k^{p}(x^{[p]_{\alpha}^{1}}-x^{[p]_{\alpha}})+(y^{[p]_{\alpha}^{1}}-y^{[p]_{\alpha}})\\
 & =k^{p}f(x)+f(y).
\end{align*}

Therefore, $f$ is $p$-semilinear.

(2)$\Rightarrow$(1) Here, we only check the axiom (R1) in Definition
1.5. For all $x,y\in L$, we have

\begin{align*}
 & \mathrm{ad}(x^{[p]_{\alpha}^{1}})\text{\textopenbullet}\alpha^{p-1}(y)\\
= & \mathrm{ad}(x^{[p]_{\alpha}}+f(x))\text{\textopenbullet}\alpha^{p-1}(y)\\
= & \mathrm{ad}(x^{[p]_{\alpha}})\text{\textopenbullet}\alpha^{p-1}(y)\\
= & \mathrm{ad}(\alpha^{p-1}(x))\text{\textopenbullet}\mathrm{ad}(\alpha^{p-2}(x))\text{\textopenbullet}\cdot\cdot\cdot\text{\textopenbullet}\mathrm{ad}(x)(y).
\end{align*}

For the check of the other two axioms given in Definition 1.5, we
can refer to \cite{GC}.
\end{proof}
Thereby, we obtain the following corollary, the proof of which is
a complete analogue of that of Corollary 2.4 in \cite{GC}.

\begin{corollary}Let $L$ be a Hom-Lie algebra. Then the following
statements hold.

(1) If $C(L)=\{x\in L|[x,\alpha^{p-1}(y)],\forall y\in L\}=0$, then
$L$ admits at most one $p$-structure.

(2) If two $p$-structures coincide on a basis, then they are equal.

(3) If $(L,[.,.]_{L},\alpha,[p]_{\alpha})$ is restricted, then there
exists a $p$-structure $[p]_{\alpha}^{1}$ on $L$ such that $x^{[p]_{\alpha}^{1}}=0$,
$\forall x\in C(L)$.

\end{corollary}

Let $U_{HLie}(L)^{-}$ be the Hom-Lie algebra given by Hom-associative
algebra $U_{HLie}(L)$ via the commutator bracket, where $U_{HLie}(L)$
is the universal enveloping algebra of Hom-Lie algebra $L$ (see \cite{Yau}).
In the special case of $g=U_{HLie}(L)^{-}\supset L$, we have

\begin{theorem} Let $(e_{j})_{j\in J}$ be a basis of $L$ such that
there are $y_{j}\in L$ with

\[
\mathrm{ad}(\alpha^{p-1}(e_{j}))\text{\textopenbullet}\mathrm{ad}(\alpha^{p-2}(e_{j}))\text{\textopenbullet}\cdot\cdot\cdot\text{\textopenbullet}\mathrm{ad}(e_{j})=\mathrm{ad}_{\alpha}(y_{j}),
\]
where $\mathrm{ad}_{\alpha}(y_{j})=\mathrm{ad}(y_{j})\text{\textopenbullet}\alpha^{p-1}$.
Then there exists exactly one $p$-structure $[p]_{\alpha}:L\rightarrow L$
such that $e_{j}^{[p]_{\alpha}}=y_{j},\forall j\in J$.

\end{theorem}
\begin{proof}
For $z\in L$, we have
\begin{align*}
0 & =\mathrm{(ad}(\alpha^{p-1}(e_{j}))\text{\textopenbullet}\mathrm{ad}(\alpha^{p-2}(e_{j}))\text{\textopenbullet}\cdot\cdot\cdot\text{\textopenbullet}\mathrm{ad}(e_{j})-\mathrm{ad}_{\alpha}(y_{j}))(z)\\
 & =[\alpha^{p-1}(e_{j}^{p})-y_{j},\alpha^{p-1}(z)].
\end{align*}

It follows that $\alpha^{p-1}(e_{j}^{p})-y_{j}\in C_{U_{HLie}(L)}(L),\forall j\in J$.
We define a $p$-semilinear mapping $f:L\rightarrow C_{U_{HLie}(L)}(L)$
by means of
\[
f\left(\sum c_{j}e_{j}\right):=\sum c_{j}^{p}(y_{j}-\alpha^{p-1}(e_{j}^{p})).
\]

Consider $V:=\{x\in L|\alpha^{p-1}(x^{p})+f(x)\in L\}$. The equation
\begin{align*}
 & \alpha^{p-1}((kx+y)^{p})+f(kx+y)\\
= & k^{p}\alpha^{p-1}(x^{p})+\alpha^{p-1}(y^{p})+\sum_{i=1}^{p-1}\alpha^{p-1}(s_{i}(kx,y))+k^{p}f(x)+f(y)
\end{align*}
ensures that $V$ is a subspace of $L$. Since $(e_{j})_{j\in J}\subset V$,
we have $\alpha^{p-1}(x^{p})+f(x)\in L,\forall x\in L$. By virtue
of Lemma 2.1, $[p]_{\alpha}:L\rightarrow L,x^{[p]_{\alpha}}:=\alpha^{p-1}(x^{p})+f(x)$
is a $p$-structure on $L$. Furthermore, we obtain $e_{j}^{[p]_{\alpha}}=\alpha^{p-1}(e_{j}^{p})+f(e_{j})=y_{j}$,
as asserted. The uniqueness of $[p]_{\alpha}$ follows from Corollary
2.2.
\end{proof}
Next, the direct sum of two restricted Hom-Lie algebras will be considered.

For two spaces $V$ and $V'$, $V\oplus V'$ denotes their direct
sum. Then the elements in the space $V\oplus V'$ have the form $x+y$,
where $x$ and $y$ are elements in $V$ and $V'$ , respectively.
We have the following result.

\begin{proposition}Given two restricted Hom-Lie algebras $(L,[.,.]_{L},\alpha,[p]_{\alpha})$
and $(L_{1},[.,.]_{L_{1}},$ $\beta,[p]_{\beta})$, there exists a
restricted Hom-Lie algebra $(L\oplus L_{1},[.,.]_{L\oplus L_{1}},\alpha+\beta,[p])$,
where the bilinear map $[.,.]_{L\oplus L_{1}}:L\oplus L_{1}\rightarrow L\oplus L_{1}$
is given by
\[
[x_{1}+y_{1},x_{2}+y_{2}]_{L\oplus L_{1}}=[x_{1},x_{2}]_{L}+[y_{1},y_{2}]_{L_{1}},\forall x_{1},x_{2}\in L,y_{1},y_{2}\in L_{1},
\]

the twisting map $\alpha+\beta:L\oplus L_{1}\rightarrow L\oplus L_{1}$
is given by
\[
(\alpha+\beta)(x+y)=\alpha(x)+\beta(y),\forall x\in L,y\in L_{1},
\]

and the $p$-structure $[p]:L\oplus L_{1}\rightarrow L\oplus L_{1}$
is given by
\[
(x+y)^{[p]}=x{}^{[p]_{\alpha}}+y{}^{[p]_{\beta}},\forall x\in L,y\in L_{1}.
\]

\end{proposition}
\begin{proof}
To check that the triple $(L\oplus L_{1},[.,.]_{L\oplus L_{1}},\alpha+\beta)$
is a Hom-Lie algebra, we can refer to \cite{GC}. The map $[p]:L\oplus L_{1}\rightarrow L\oplus L_{1}$
will be proved to be a $p$-structure on $L\oplus L_{1}$. Indeed,
for all $x_{1}+y_{1},x_{2}+y_{2}\in L\oplus L_{1}$, we have
\begin{align*}
 & \mathrm{ad}((x_{1}+y_{1})^{[p]})\text{\textopenbullet}(\alpha+\beta)^{p-1}(x_{2}+y_{2})\\
= & [(x_{1}+y_{1})^{[p]},\alpha{}^{p-1}(x_{2})+\beta{}^{p-1}(y_{2})]_{L\oplus L_{1}}\\
= & [x_{1}^{[p]_{\alpha}}+y_{1}^{[p]_{\beta}},\alpha{}^{p-1}(x_{2})+\beta{}^{p-1}(y_{2})]_{L\oplus L_{1}}\\
= & [x_{1}^{[p]_{\alpha}},\alpha{}^{p-1}(x_{2})]_{L}+[y_{1}^{[p]_{\beta}},\beta{}^{p-1}(y_{2})]_{L_{1}}\\
= & \mathrm{ad}x_{1}^{[p]_{\alpha}}\lyxmathsym{\textopenbullet}\alpha^{p-1}(x_{2})+\mathrm{ad}y_{1}^{[p]_{\beta}}\text{\textopenbullet}\beta^{p-1}(y_{2})\\
= & \mathrm{ad}(\alpha^{p-1}(x_{1}))\text{\textopenbullet}\mathrm{ad}(\alpha^{p-2}(x_{1}))\text{\textopenbullet}\cdot\cdot\cdot\text{\textopenbullet}\mathrm{ad}(x_{1})(x_{2})\\
 & +\mathrm{ad}(\beta^{p-1}(y_{1}))\text{\textopenbullet}\mathrm{ad}(\beta^{p-2}(y_{1}))\text{\textopenbullet}\cdot\cdot\cdot\text{\textopenbullet}\mathrm{ad}(y_{1})(y_{2})
\end{align*}

and
\begin{align*}
 & \mathrm{ad}((\alpha+\beta)^{p-1}(x_{1}+y_{1}))\text{\textopenbullet}\mathrm{ad}((\alpha+\beta)^{p-2}(x_{1}+y_{1}))\text{\textopenbullet}\cdot\cdot\cdot\text{\textopenbullet}\mathrm{ad}(x_{1}+y_{1})(x_{2}+y_{2})\\
= & \mathrm{ad}(\alpha^{p-1}(x_{1})+\beta{}^{p-1}(y_{1}))\text{\textopenbullet}\mathrm{ad}(\alpha{}^{p-2}(x_{1}\text{)}+\beta{}^{p-2}(y_{1}))\text{\textopenbullet}\cdot\cdot\cdot\text{\textopenbullet}\mathrm{ad}(x_{1}+y_{1})(x_{2}+y_{2})\\
= & \mathrm{ad}(\alpha^{p-1}(x_{1}))\text{\textopenbullet}\mathrm{ad}(\alpha^{p-2}(x_{1}))\text{\textopenbullet}\cdot\cdot\cdot\text{\textopenbullet}\mathrm{ad}(x_{1})(x_{2})\\
 & +\mathrm{ad}(\beta^{p-1}(y_{1}))\text{\textopenbullet}\mathrm{ad}(\beta^{p-2}(y_{1}))\text{\textopenbullet}\cdot\cdot\cdot\text{\textopenbullet}\mathrm{ad}(y_{1})(y_{2}).
\end{align*}

Therefore, we obtain
\begin{align*}
 & \mathrm{ad}((x_{1}+y_{1})^{[p]})\text{\textopenbullet}(\alpha+\beta)^{p-1}(x_{2}+y_{2})\\
= & \mathrm{ad}((\alpha+\beta)^{p-1}(x_{1}+y_{1}))\text{\textopenbullet}\mathrm{ad}((\alpha+\beta)^{p-2}(x_{1}+y_{1}))\text{\textopenbullet}\cdot\cdot\cdot\text{\textopenbullet}\mathrm{ad}(x_{1}+y_{1})(x_{2}+y_{2}).
\end{align*}

Moreover, for all $k\in\mathbb{F},x+y\in L\oplus L_{1}$, one gets
\begin{align*}
 & (k(x+y))^{[p]}=(kx+ky)^{[p]}\\
= & (kx)^{[p]_{\alpha}}+(ky)^{[p]_{\beta}}=k^{p}x{}^{[p]_{\alpha}}+k^{p}y{}^{[p]_{\beta}}=k^{p}(x+y){}^{[p]}.
\end{align*}

Finally, for all $x_{1}+y_{1},x_{2}+y_{2}\in L\oplus L_{1}$, we have
\begin{align*}
 & ((x_{1}+y_{1})+(x_{2}+y_{2}))^{[p]}\\
= & ((x_{1}+x_{2})+(y_{1}+y_{2}))^{[p]}=(x_{1}+x_{2})^{[p]_{\alpha}}+(y_{1}+y_{2}){}^{[p]_{\beta}}\\
= & x_{1}^{[p]_{\alpha}}+x_{2}^{[p]_{\alpha}}+\sum_{i=1}^{p-1}s_{i}(x_{1},x_{2})+y_{1}^{[p]_{\beta}}+y_{2}^{[p]_{\beta}}+\sum_{i=1}^{p-1}s_{i}(y_{1},y_{2})\\
= & (x_{1}^{[p]_{\alpha}}+y_{1}^{[p]_{\beta}})+(x_{2}^{[p]_{\alpha}}+y_{2}^{[p]_{\beta}})+\sum_{i=1}^{p-1}s_{i}(x_{1}+y_{1},x_{2}+y_{2})\\
= & (x_{1}+y_{1})^{[p]}+(x_{2}+y_{2})^{[p]}+\sum_{i=1}^{p-1}s_{i}(x_{1}+y_{1},x_{2}+y_{2}).
\end{align*}

It follows that $[p]$ is a $p$-structure on $L\oplus L_{1}$.

To sum up, $(L\oplus L_{1},[.,.]_{L\oplus L_{1}},\alpha+\beta,[p])$
is a restricted Hom-Lie algebra.
\end{proof}
For further discussion, let's recall the definitions of a (restricted)
morphism of a Hom-Lie algebra first.

\begin{definition} (see \cite{She}) Let $(L,[\cdot,\cdot]_{L},\alpha)$
and $(L_{1},[\cdot,\cdot]_{L_{1}},\alpha_{1})$ be two Hom-Lie algebras.
A linear map $f:L\rightarrow L_{1}$ is said to be a morphism of Hom-Lie
algebras if
\begin{align*}
f[x,y]_{L} & =[f(x),f(y)]_{L_{1}},\forall x,y\in L,\\
f\text{\textopenbullet}\alpha & =\alpha_{1}\text{\textopenbullet}f.
\end{align*}

\end{definition}

\begin{definition}(see \cite{GC}) A morphism of restricted Hom-Lie
algebras
\[
f:(L,[\cdot,\cdot]_{L},\alpha,[p]_{\alpha})\rightarrow(L_{1},[\cdot,\cdot]_{L_{1}},\alpha_{1},[p]_{\alpha_{1}})
\]
is said to be restricted if $f(x^{[p]_{\alpha}})=(f(x))^{[p]_{\alpha_{1}}}$,
for all $x\in L$.

\end{definition}

For a linear map $\phi:L\rightarrow L'$, let $\mathscr{\mathfrak{G}}_{\phi}=\{(x,\phi(x))|x\in L\}$
denote the graph of $\phi$. We obtain the following proposition,
which proof is analogous to that of Proposition 3.4 in \cite{GC}:

\begin{proposition} A linear map $\phi:(L,[\cdot,\cdot]_{L},\alpha,[p]_{\alpha})\rightarrow(L_{1},[\cdot,\cdot]_{L_{1}},\beta,[p]_{\beta})$
is a restricted morphism of restricted Hom-Lie algebras if and only
if the graph $\mathfrak{G}_{\phi}\subseteq L\oplus L_{1}$ is a restricted
Hom-Lie subalgebra of $(L\oplus L_{1},[\cdot,\cdot]_{L\oplus L_{1}},\alpha+\beta,[p])$,
which is a Hom-Lie algebra defined as in Proposition 2.4.

\end{proposition}

\begin{theorem} Let $(L,[\cdot,\cdot]_{L},\alpha,[p]_{\alpha})$
be a restricted Hom-Lie algebra, $(L_{1},[\cdot,\cdot]_{L_{1}},\beta)$
a multiplicative Hom-Lie algebra, and $f:L\rightarrow L_{1}$ a morphism
of Hom-Lie algebras. Then $(f(L),[\cdot,\cdot]_{L_{1}},\beta|_{f(L)},[p]_{\beta})$
is a restricted Hom-Lie algebra, where $[p]_{\beta}:f(L)\rightarrow f(L),f(x)\mapsto f(x^{[p]_{\alpha}})$.

\end{theorem}
\begin{proof}
For $f(x),f(y)\in f(L)$, we have $[f(x),f(y)]_{L_{1}}=f([x,y]_{L})\in f(L)$
and $\beta(f(x))=f(\alpha(x))\in f(L)$. Therefore, $(f(L),[\cdot,\cdot]_{L_{1}},\beta|_{f(L)})$
is a Hom-Lie subalgebra of $L_{1}$. Next, the map $[p]_{\beta}:f(L)\rightarrow f(L)$
will be proved to be a $p$-structure on $f(L)$ and $(f(L),[\cdot,\cdot]_{L_{1}},$
$\beta|_{f(L)},[p]_{\beta})$ is restricted.

The check of axioms (R2) and (R3) given in Definition $1.5$ is a
routine. We will show the verification of axiom (R1) detailedly. For
$f(x),f(y)\in f(L)$, one gets
\begin{align*}
 & \mathrm{ad}(f(x)^{[p]_{\beta}})\text{\textopenbullet}\beta^{p-1}(f(y))=[f(x)^{[p]_{\beta}},\beta^{p-1}(f(y))]_{L_{1}}\\
= & [f(x^{[p]_{\alpha}}),f(\alpha^{p-1}(y))]_{L_{1}}=f([x^{[p]_{\alpha}},\alpha^{p-1}(y)]_{L})=f(\mathrm{ad}(x^{[p]_{\alpha}})\circ\alpha^{p-1}(y))\\
= & f(\mathrm{ad}(\alpha^{p-1}(x))\text{\textopenbullet}\mathrm{ad}(\alpha^{p-2}(x))\text{\textopenbullet}\cdot\cdot\cdot\text{\textopenbullet}\mathrm{ad}(x)(y))\\
= & f([\alpha^{p-1}(x),[\alpha^{p-2}(x),\cdot\cdot\cdot,[x,y]_{L}]_{L}]_{L})\\
= & [f(\alpha^{p-1}(x)),f([\alpha^{p-2}(x),\cdot\cdot\cdot,[x,y]_{L}]_{L})]_{L_{1}}\\
= & [f(\alpha^{p-1}(x)),[f(\alpha^{p-2}(x)),\cdot\cdot\cdot,[f(x),f(y)]_{L_{1}}]_{L_{1}}]_{L_{1}}\\
= & [\beta^{p-1}(f(x)),[\beta^{p-2}(f(x)),\cdot\cdot\cdot,[f(x),f(y)]_{L_{1}}]_{L_{1}}]_{L_{1}}\\
= & \mathrm{ad}(\beta^{p-1}(f(x)))\circ\mathrm{ad}(\beta^{p-2}(f(x)))\circ\cdot\cdot\cdot\circ\mathrm{ad}(f(x))(f(y)),
\end{align*}
that is, axiom (R1) holds. The proof is complete.
\end{proof}
\begin{theorem} Let $(L,[\cdot,\cdot]_{L},\alpha)$ and $(L_{1},[\cdot,\cdot]_{L_{1}},\beta)$
be multiplicative Hom-Lie algebras and $f:L\rightarrow L_{1}$ a morphism
of Hom-Lie algebras. If $g$ is a Hom-Lie subalgebra of $L_{1}$ and
$[p]_{\beta}:g\rightarrow g$ is a $p$-structure, then $(f^{-1}(g),[\cdot,\cdot]_{L},\alpha|_{f^{-1}(g)})$
is a Hom-Lie subalgebra of $L$ and $[p]_{\alpha}:f^{-1}(g)\rightarrow f^{-1}(g),x\mapsto f^{-1}(f(x)^{[p]_{\beta}})$
is a $p$-structure on $f^{-1}(g)$.

\end{theorem}
\begin{proof}
Firstly, $f^{-1}(g)$ will be shown to be a Hom-Lie subalgebra of
$L$. For $x,y\in f^{-1}(g)$, one gets $f(x),f(y)\in g$. It follows
that $f([x,y]_{L})=[f(x),f(y)]_{L_{1}}\in g$ and $f(\alpha(x))=\beta(f(x))\in g$.
Therefore, $[x,y]_{L}\in f^{-1}(g)$ and $\alpha(x)\in f^{-1}(g)$.

Secondly, we will show that $[p]_{\alpha}:f^{-1}(g)\rightarrow f^{-1}(g)$
is a $p$-structure. Indeed, for $x,y\in f^{-1}(g)$, there exist
$u,v\in g$ such that $f(x)=u,f(y)=v$. Therefore,
\begin{align*}
 & \mathrm{ad}(x^{[p]_{\alpha}})\text{\textopenbullet}\alpha^{p-1}(y)=[x^{[p]_{\alpha}},\alpha^{p-1}(y)]_{L}\\
= & [f^{-1}(f(x)^{[p]_{\beta}}),f^{-1}(f(\alpha^{p-1}(y)))]_{L}=f^{-1}([f(x)^{[p]_{\beta}},f(\alpha^{p-1}(y))]_{L_{1}})\\
= & f^{-1}([f(x)^{[p]_{\beta}},\beta^{p-1}(f(y))]_{L_{1}})=f^{-1}([u^{[p]_{\beta}},\beta^{p-1}(v)]_{L_{1}})\\
= & f^{-1}(\mathrm{ad}(u^{[p]_{\beta}})\circ\beta^{p-1}(v))=f^{-1}(\mathrm{ad}(\beta^{p-1}(u))\text{\textopenbullet}\mathrm{ad}(\beta^{p-2}(u))\text{\textopenbullet}\cdot\cdot\cdot\text{\textopenbullet}\mathrm{ad}(u)(v))\\
= & f^{-1}([\beta^{p-1}(u),[\beta^{p-2}(u),\cdot\cdot\cdot,[u,v]_{L_{1}}]_{L_{1}}]_{L_{1}})\\
= & [f^{-1}(\beta^{p-1}(u)),f^{-1}([\beta^{p-2}(u),\cdot\cdot\cdot,[u,v]_{L_{1}}]_{L_{1}})]_{L}\\
= & [f^{-1}(\beta^{p-1}(u)),[f^{-1}(\beta^{p-2}(u)),\cdot\cdot\cdot,[f^{-1}(u),f^{-1}(v)]_{L}]_{L}]_{L}\\
= & [\alpha^{p-1}(f^{-1}(u)),[\alpha^{p-2}(f^{-1}(u)),\cdot\cdot\cdot,[f^{-1}(u),f^{-1}(v)]_{L}]_{L}]_{L}\\
= & [\alpha^{p-1}(x),[\alpha^{p-2}(x),\cdot\cdot\cdot,[x,y]_{L}]_{L}]_{L}\\
= & \mathrm{ad}(\alpha^{p-1}(x))\text{\textopenbullet}\mathrm{ad}(\alpha^{p-2}(x))\text{\textopenbullet}\cdot\cdot\cdot\text{\textopenbullet}\mathrm{ad}(x)(y).
\end{align*}

Until now, the check of axiom (R1) is complete. Moreover, the check
of axioms (R2) and (R3) is a routine, see \cite{Sha} for more details.
The proof is complete.
\end{proof}

\section{Restrictable Hom-Lie Algebras}

\begin{definition}

A multiplicative Hom-Lie algebra $(L,[.,.]_{L},\alpha)$ is said to
be restrictable if
\[
\mathrm{ad}(\alpha^{p-1}(x))\text{\textopenbullet}\mathrm{ad}(\alpha^{p-2}(x))\text{\textopenbullet}\cdot\cdot\cdot\text{\textopenbullet}\mathrm{ad}(x)\in\mathrm{ad}_{\alpha}(L),
\]
for all $x\in L$, where $\mathrm{ad_{\alpha}}L=\{\mathrm{ad}_{\alpha}(x)|x\in L$\}.

\end{definition}

Several examples are shown here.

\begin{example} Any abelian Hom-Lie algebra is restrictable.

\end{example}

\begin{example} Let $(L,[.,.]_{L},\alpha)$ be a nilpotent Hom-Lie
algebra such that $L^{p+1}=0$. Then $\mathrm{ad}(\alpha^{p-1}(x))\text{\textopenbullet}\mathrm{ad}(\alpha^{p-2}(x))\text{\textopenbullet}\cdot\cdot\cdot\text{\textopenbullet}\mathrm{ad}(x)=0$,
$\forall x\in L$. If $\mathrm{dim}L<p+2$, then $(L,[.,.]_{L},\alpha)$
is restrictable.

\end{example}

\begin{example} Consider the two-dimensional nonabelian Lie algebra
$L:=\mathbb{F}h\oplus\mathbb{F}x$ with the nonzero bracket $[h,x]=x$.
Let $\alpha:L\rightarrow L$ be a mapping defined by $\alpha(h)=h,\alpha(x)=0$
and $[\cdot,\cdot]_{\alpha}=\alpha\circ[\cdot,\cdot]$. Then the triple
$(L,[.,.]_{\alpha},\alpha)$ is a Hom-Lie algebra. Furthermore, $\mathrm{ad}(\alpha^{p-1}(h))\text{\textopenbullet}\mathrm{ad}(\alpha^{p-2}(h))\text{\textopenbullet}\cdot\cdot\cdot\text{\textopenbullet}\mathrm{ad}(h)=\mathrm{ad}_{\alpha}(h)\in\mathrm{ad}_{\alpha}(L)$
and $\mathrm{ad}(\alpha^{p-1}(x))\text{\textopenbullet}\mathrm{ad}(\alpha^{p-2}(x))\text{\textopenbullet}\cdot\cdot\cdot\text{\textopenbullet}\mathrm{ad}(x)=\mathrm{ad}_{\alpha}(x)\in\mathrm{ad}_{\alpha}(L)$.
It implies that $(L,[.,.]_{\alpha},\alpha)$ is restrictable.

\end{example}

In fact, there exists a equivalence relation between restrictable
Hom-Lie algebras and restricted Hom-Lie algebras. The details are
as follows:

\begin{theorem} A multiplicative Hom-Lie algebra $(L,[.,.]_{L},\alpha)$
is restrictable if and only if there exists a $p$-structure $[p]_{\alpha}:L\rightarrow L$
which makes $L$ a restricted Hom-Lie algebra.

\end{theorem}
\begin{proof}
$(\Leftarrow)$ Suppose that $L$ is restricted. For any $x\in L$,
we have
\begin{align*}
 & \mathrm{ad}(\alpha^{p-1}(x))\text{\textopenbullet}\mathrm{ad}(\alpha^{p-2}(x))\text{\textopenbullet}\cdot\cdot\cdot\text{\textopenbullet}\mathrm{ad}(x)\\
= & \mathrm{ad}(x^{[p]_{\alpha}})\text{\textopenbullet}\alpha^{p-1}\\
= & \mathrm{ad}_{\alpha}(x^{[p]_{\alpha}})\in\mathrm{ad_{\alpha}}L,
\end{align*}

which proves that $L$ is restrictable.

$(\Rightarrow)$ Let $L$ be restrictable. Then for $x\in L$, we
obtain
\[
\mathrm{ad}(\alpha^{p-1}(x))\text{\textopenbullet}\mathrm{ad}(\alpha^{p-2}(x))\text{\textopenbullet}\cdot\cdot\cdot\text{\textopenbullet}\mathrm{ad}(x)\in\mathrm{ad}_{\alpha}(L),
\]

that is, there exists $y\in L$ such that
\[
\mathrm{ad}(\alpha^{p-1}(x))\text{\textopenbullet}\mathrm{ad}(\alpha^{p-2}(x))\text{\textopenbullet}\cdot\cdot\cdot\text{\textopenbullet}\mathrm{ad}(x)=\mathrm{ad}_{\alpha}(y).
\]

Let $(e_{j})_{j\in J}$ be a basis of $L$. Then there exist $y_{j}\in L$
such that
\[
\mathrm{ad}(\alpha^{p-1}(e_{j}))\text{\textopenbullet}\mathrm{ad}(\alpha^{p-2}(e_{j}))\text{\textopenbullet}\cdot\cdot\cdot\text{\textopenbullet}\mathrm{ad}(e_{j})=\mathrm{ad}_{\alpha}(y_{j}),\forall j\in J.
\]

By virtue of Theorem 2.3, there exists exactly one $p$-structure
$[p]_{\alpha}:L\rightarrow L$ such that $e_{j}^{[p]_{\alpha}}=y_{j},\forall j\in J$,
which makes $L$ a restricted Hom-Lie algebra.
\end{proof}
Clearly, the notion of restrictable Hom-Lie algebras is more tractable
than that of restricted Hom-Lie algebras. Moreover, it is required
only to know whether a $p$-structure can be defined on a given Hom-Lie
algebra in many cases. Therefore, we only need to determine whether
a given Hom-Lie algebra is restrictable instead of defining a $p$-structure
on it.

One advantage in considering restrictable Hom-Lie algebras instead
of restricted ones reflects in the following theorem:

\begin{theorem} Let $f:(L,[\cdot,\cdot]_{L},\alpha)\rightarrow(L_{1},[\cdot,\cdot]_{L_{1}},\alpha_{1})$
be a surjective morphism of Hom-Lie algebras. If $L$ is restrictable,
so is $L_{1}$.

\end{theorem}
\begin{proof}
Since $f$ is surjective, we have $L_{1}=f(L)$, that is, for any
$z\in L_{1}$, there exists $x\in L$ such that $f(x)=z$. For $x\in L$,
it follows from $L$ being restrictable that there is $y\in L$ such
that
\[
\mathrm{ad}(\alpha^{p-1}(x))\text{\textopenbullet}\mathrm{ad}(\alpha^{p-2}(x))\text{\textopenbullet}\cdot\cdot\cdot\text{\textopenbullet}\mathrm{ad}(x)=\mathrm{ad}_{\alpha}(y)\text{.}
\]

For $w\in L$, one gets
\begin{align*}
 & \mathrm{ad}(\alpha_{1}^{p-1}(f(x)))\text{\textopenbullet}\mathrm{ad}(\alpha_{1}^{p-2}(f(x)))\text{\textopenbullet}\cdot\cdot\cdot\text{\textopenbullet}\mathrm{ad}(f(x))(f(w))\\
= & [\alpha_{1}^{p-1}(f(x)),[\alpha_{1}^{p-2}(f(x)),\cdot\cdot\cdot,[f(x),f(w)]_{L_{1}}]_{L_{1}}]_{L_{1}}\\
= & [f(\alpha^{p-1}(x)),[f(\alpha^{p-2}(x)),\cdot\cdot\cdot,[f(x),f(w)]_{L_{1}}]_{L_{1}}]_{L_{1}}\\
= & f([\alpha^{p-1}(x),[\alpha^{p-2}(x),\cdot\cdot\cdot,[x,w]_{L}]_{L}]_{L}\text{)}\\
= & f(\mathrm{ad}(\alpha^{p-1}(x))\text{\textopenbullet}\mathrm{ad}(\alpha^{p-2}(x))\text{\textopenbullet}\cdot\cdot\cdot\text{\textopenbullet}\mathrm{ad}(x)(w\text{))}\\
= & f(\mathrm{ad}_{\alpha}(y)(w))=f([y,\alpha^{p-1}(w)]_{L})\\
= & [f(y),f(\alpha^{p-1}(w))]_{L_{1}}=[f(y),\alpha_{1}^{p-1}(f(w))]_{L_{1}}\\
= & \mathrm{ad}_{\alpha_{1}}(f(y))(f(w)),
\end{align*}
which implies that $\mathrm{ad}(\alpha_{1}^{p-1}(f(x)))\text{\textopenbullet}\mathrm{ad}(\alpha_{1}^{p-2}(f(x)))\text{\textopenbullet}\cdot\cdot\cdot\text{\textopenbullet}\mathrm{ad}(f(x))=\mathrm{ad}_{\alpha_{1}}(f(y))\in\mathrm{ad}_{\alpha_{1}}L_{1}$.
Hence $L_{1}$ is restrictable.
\end{proof}
\begin{theorem} Let $L_{1}$ and $L_{2}$ be Hom-Lie ideals of Hom-Lie
algebra $(L,[\cdot,\cdot]_{L},\alpha)$ such that $L=L_{1}\oplus L_{2}$.
Then $L$ is restrictable if and only if $L_{1},L_{2}$ are restrictable.

\end{theorem}
\begin{proof}
($\Leftarrow$) Assuming $L_{1},L_{2}$ are restrictable. For $x\in L$,
we may suppose that $x=x_{1}+x_{2}$, where $x_{1}\in L_{1},x_{2}\in L_{2}$.
Since $L_{1},L_{2}$ are Hom-Lie ideals, one gets $\alpha^{i}(x_{1})\in L_{1},\alpha^{j}(x_{2})\in L_{2}$,
for $i,j\in\{1,2,...,p-1\}$. As $L_{1},L_{2}$ are restrictable,
there exist $y_{1}\in L_{1},y_{2}\in L_{2}$ such that $\mathrm{ad}(\alpha^{p-1}(x_{1}))\text{\textopenbullet}\mathrm{ad}(\alpha^{p-2}(x_{1}))\text{\textopenbullet}\cdot\cdot\cdot\text{\textopenbullet}\mathrm{ad}(x_{1})=\mathrm{ad}_{\alpha}(y_{1})$
and $\mathrm{ad}(\alpha^{p-1}(x_{2}))\text{\textopenbullet}\mathrm{ad}(\alpha^{p-2}(x_{2}))\text{\textopenbullet}\cdot\cdot\cdot\text{\textopenbullet}\mathrm{ad}(x_{2})=\mathrm{ad}_{\alpha}(y_{2})$
respectively. Therefore,
\begin{align*}
 & \mathrm{ad}(\alpha^{p-1}(x))\text{\textopenbullet}\mathrm{ad}(\alpha^{p-2}(x))\text{\textopenbullet}\cdot\cdot\cdot\text{\textopenbullet}\mathrm{ad}(x)\\
= & \mathrm{ad}(\alpha^{p-1}(x_{1}+x_{2}))\text{\textopenbullet}\mathrm{ad}(\alpha^{p-2}(x_{1}+x_{2}))\text{\textopenbullet}\cdot\cdot\cdot\text{\textopenbullet}\mathrm{ad}(x_{1}+x_{2})\\
= & \mathrm{ad}(\alpha^{p-1}(x_{1}))\text{\textopenbullet}\mathrm{ad}(\alpha^{p-2}(x_{1}))\text{\textopenbullet}\cdot\cdot\cdot\text{\textopenbullet}\mathrm{ad}(x_{1})\\
 & \mathrm{+ad}(\alpha^{p-1}(x_{2}))\text{\textopenbullet}\mathrm{ad}(\alpha^{p-2}(x_{2}))\text{\textopenbullet}\cdot\cdot\cdot\text{\textopenbullet}\mathrm{ad}(x_{2})\\
= & \mathrm{ad}_{\alpha}(y_{1})+\mathrm{ad}_{\alpha}(y_{2})\\
= & \mathrm{ad}_{\alpha}(y_{1}+y_{2})\in\mathrm{ad}_{\alpha}L.
\end{align*}

Hence $L$ is restrictable.

($\Rightarrow$) Suppose that $L$ is restrictable. Theorem 3.6 applies
and $L_{1}\cong L/L_{2},L_{2}\cong L/L_{1}$ are restrictable.
\end{proof}
\begin{corollary}Let $L_{1},L_{2}$ be restrictable Hom-Lie ideals
of Hom-Lie algebra $(L,[\cdot,\cdot]_{L},\alpha)$ such that $L=L_{1}+L_{2}$
and $[L_{1},L_{2}]=\{0\}$. Then $L$ is restrictable.

\end{corollary}
\begin{proof}
By virtue of Theorem 3.7, since $L_{1},L_{2}$ are restrictable, one
gets $L_{1}\oplus L_{2}$ is restrictable. We define a mapping $f:L_{1}\oplus L_{2}\rightarrow L,(x,y)\mapsto x+y$.
Obviously, $f$ is surjective. Furthermore, be analogous to Corollary
3.7 in \cite{GC}, it can be proved that $f$ is a morphism of Hom-Lie
algebras. Theorem 3.6 applies and $L$ is restrictable.
\end{proof}
\begin{corollary}Suppose that $(L,[.,.]_{L},\alpha)$ is a regular
Hom-Lie algebra with $L=[L,L]_{L}+C(L)$. Let $I\subset C(L)$ be
a Hom-Lie ideal of $L$. Then $L$ is restrictable if and only if
$L/I$ is restrictable.

\end{corollary}
\begin{proof}
($\Rightarrow$) Assume that $L$ is restrictable. Then Theorem 3.6
applies and $L/I$ is restrictable.

($\Leftarrow$) If $L/I$ is restrictable, then for any $x\in[L,L]_{L}$,
there exists $y\in[L,L]_{L}$ such that
\[
\mathrm{(ad}(\alpha^{p-1}(x))\circ\mathrm{ad}(\alpha^{p-2}(x))\circ\cdot\cdot\cdot\circ\mathrm{ad}(x)-\mathrm{ad}_{\alpha}y)(L)\subset I\subset C(L).
\]

Therefore,
\[
\mathrm{(ad}(\alpha^{p-1}(x))\circ\mathrm{ad}(\alpha^{p-2}(x))\circ\cdot\cdot\cdot\circ\mathrm{ad}(x)-\mathrm{ad}_{\alpha}y)([L,L]_{L})\subset[I,L]_{L}\subset[C(L),L]_{L}=\{0\text{\}}.
\]

It implies that
\[
\mathrm{ad}(\alpha^{p-1}(x))\circ\mathrm{ad}(\alpha^{p-2}(x))\circ\cdot\cdot\cdot\circ\mathrm{ad}(x)=\mathrm{ad}_{\alpha}y\in\mathrm{ad}_{\alpha}([L,L]_{L}).
\]

Hence $[L,L]_{L}$ is restrictable. Moreover, $C(L)$ is clearly restrictable
by definition. Together with $[[L,L]_{L},C(L)]_{L}=\{0\}$, we conclude
by Corollary 3.8 that $L$ is restrictable.
\end{proof}
\begin{proposition}Let $(L,[.,.]_{L},\alpha)$ be a restrictable
Hom-Lie algebra and $H$ a Hom-Lie subalgebra of $L$. Then $H$ is
a $p$-subalgebra for some $p$-structure $[p]_{\alpha}$ on $L$
if and only if $(\mathrm{ad}H|_{L})^{p}\subset\mathrm{ad}_{\alpha}H|_{L}$.

\end{proposition}
\begin{proof}
($\Rightarrow$) Assume that $H$ is a $p$-subalgebra of $L$, then
for any $x\in H$, one gets $x^{[p]_{\alpha}}\in H$. Therefore, $\mathrm{ad}(\alpha^{p-1}(x))\text{\textopenbullet}\mathrm{ad}(\alpha^{p-2}(x))\text{\textopenbullet}\cdot\cdot\cdot\text{\textopenbullet}\mathrm{ad}(x)=\mathrm{ad}_{\alpha}(x^{[p]_{\alpha}})\in\mathrm{ad}_{\alpha}H|_{L}$.
It implies that $(\mathrm{ad}H|_{L})^{p}\subset\mathrm{ad}_{\alpha}H|_{L}$.

($\Leftarrow$) If $(\mathrm{ad}H|_{L})^{p}\subset\mathrm{ad}_{\alpha}H|_{L}$,
then $H$ is restrictable. Theorem 3.5 applies and $H$ is restricted.
Thereby, $H$ is a $p$-subalgebra of $L$.
\end{proof}
At the end of this section, the influence of associative bilinear
forms will be considered. To this end, the following definitions are
introduced first.

\begin{definition}Let $(L,[.,.]_{L},\alpha)$ be a restricted Hom-Lie
algebra in characteristic $p$. A symmetric bilinear form $\lambda:L\times L\rightarrow\mathbb{F}$
is called associative, if
\[
\lambda(x,[y,z])=\lambda([x,\alpha^{p-1}(y)],z),\forall x,y,z\in L.
\]

\end{definition}

\begin{definition} (see \cite{GC}) Let $(L,[.,.]_{L},\alpha)$ be
a Hom-Lie algebra and $\lambda$ a symmetric bilinear form on $L$.
Set $L^{\perp}=\{x\in L|\lambda(x,y)=0,\forall y\in L\}$. $L$ is
called nondegenerate if $L^{\perp}=0$.

\end{definition}

\begin{theorem} Let $(g,[.,.]_{g},\alpha,[p]_{\alpha})$ be a restricted
Hom-Lie algebra and $L$ a Hom-Lie subalgebra of $g$ with $C(L)=\{0\}$.
Assume $\lambda:g\times g\rightarrow\mathbb{F}$ to be an associative
symmetric bilinear form, which is nondegenerate on $L\times L$. Then
$L$ is restrictable.

\end{theorem}
\begin{proof}
Since $\lambda$ is nondegenerate on $L\times L$, every linear form
$f$ on $L$ is determined by a suitably chosen element $y\in L:f(z)=\lambda(y,z),\forall z\in L$.
Then there exists $y\in L$ such that
\[
\lambda(x^{[p]_{\alpha}},z)=\lambda(y,z),\forall z\in L,
\]
which implies that $0=\lambda(x^{[p]_{\alpha}}-y,[L,L])=\lambda([x^{[p]_{\alpha}}-y,\alpha^{p-1}(L)],L)$.
Thereby, $[x^{[p]_{\alpha}}-y,\alpha^{p-1}(L)]=0$. It follows from
$C(L)=\{0\}$ that $x^{[p]_{\alpha}}-y=0$ and $x^{[p]_{\alpha}}=y$.
Hence
\begin{align*}
 & \mathrm{ad}(\alpha^{p-1}(x))\text{\textopenbullet}\mathrm{ad}(\alpha^{p-2}(x))\text{\textopenbullet}\cdot\cdot\cdot\text{\textopenbullet}\mathrm{ad}(x)\\
= & \mathrm{ad}_{\alpha}(x^{[p]_{\alpha}})=\mathrm{ad}_{\alpha}(y)\in\mathrm{ad}_{\alpha}L,
\end{align*}
which proves that $L$ is restrictable.
\end{proof}

\section{Restricted Envelopes}

In this section, the universal enveloping algebras and $p$-envelopes
of a restricted Hom-Lie algebra will be taken into consideration.
Like the case of restricted Lie algebras, the additional structure
of restrictedness should be considered.

\begin{definition}Let $(L,[.,.]_{L},\alpha,[p]_{\alpha})$ denote
a restricted Hom-Lie algebra. Let $u_{HLie}(L)$ be a Hom-associative
algebra with unity and $i:L\rightarrow u_{HLie}(L)^{-}$ a restricted
morphism. The pair $(u_{HLie}(L),i)$ is called a restricted universal
enveloping algebra if for every Hom-associative algebra $(A,\mu_{A},\alpha_{A})$
and every restricted morphism $f:L\rightarrow HLie(A)$ of Hom-Lie
algebras, there exists a unique morphism $h:u_{HLie}(L)\rightarrow A$
of Hom-associative algebras such that $h\circ i=f$.

\end{definition}

\begin{definition}Let $(L,[.,.]_{L},\alpha_{L})$ be any Hom-Lie
algebra.

(1) The 5-tuple $(G,[.,.]_{G},\alpha_{G},[p],i)$ consisting of a
restricted Hom-Lie algebra $(G,[.,.]_{G},$ $\alpha_{G},[p])$ and
a morphism of Hom-Lie algebras $i:L\rightarrow G$ is called a $p$-envelope
of $L$ if $i$ is injective and $i(L)_{p}=G$.

(2) A $p$-envelope $(G,[.,.]_{G},\alpha_{G},[p],i)$ is said to be
universal if it satisfies the following universal property: For every
restricted Hom-Lie algebra $(H,[.,.]_{H},\alpha_{H},[p]')$ and every
morphism $f:L\rightarrow H$, there exists exactly one restricted
morphism of Hom-Lie algebras $g:(G,[.,.]_{G},\alpha_{G},[p])\rightarrow(H,[.,.]_{H},\alpha_{H},[p]')$
such that $g\circ i=f$.

\end{definition}

The existence of universal $p$-envelopes is ensured by the following
theorem:

\begin{theorem} Every Hom-Lie algebra $(L,[.,.]_{L},\alpha_{L})$
has a universal $p$-envelope $\hat{L}$. Moreover, if $(L,[.,.]_{L},\alpha_{L})$
is involutive, then $(\hat{L},[.,.]_{\hat{L}},\alpha_{\hat{L}})$
is also an involutive Hom-Lie algebra.

\end{theorem}
\begin{proof}
Let $\hat{L}$ be the $p$-subalgebra of $U_{HLie}(L)^{-}$ generated
by $L$ and $f:L\rightarrow H$ a morphism. Be analogous to the case
of Lie algebra(see \cite{SF}), $\hat{f}:\hat{L}\rightarrow H$ can
be proved to be the unique extension of $f$. Therefore, $\hat{L}$
is a universal $p$-envelope of $L$. Moreover, if $\alpha_{L}:L\rightarrow L$
is an involution, then $\alpha_{\hat{L}}:\hat{L}\rightarrow\hat{L}$
is obviously also an involution. It follows that $(\hat{L},[.,.]_{\hat{L}},\alpha_{\hat{L}})$
is an involutive Hom-Lie algebra.
\end{proof}
\begin{proposition} Let $(L,[.,.]_{L},\alpha)$ be an involutive
Hom-Lie algebra. Then the following statements hold:

(1) Let $(G,[.,.]_{G},\alpha_{G},[p],i)$ be a universal $p$-envelope
of $L$. If $L$ is finite-dimensional, so is $G/C(G)$.

(2) If $L$ is finite-dimensional, then $L$ possesses a finite-dimensional
$p$-envelope.

(3) Any morphism of Hom-Lie algebra $f:L\rightarrow G$ can be extended
to a restricted morphism $\hat{f}:\hat{L}\rightarrow\hat{G}$. If
$f$ is surjective, so is $\hat{f}$.

\end{proposition}
\begin{proof}
(1) As $G=L_{p}$, we have $[G,G]\subset L$ and $L$ be a Hom-Lie
ideal of $G$. Let $L_{1}=\{x\in L|\alpha(x)=x\}$. Consider the morphism
$\phi:G\rightarrow\mathrm{ad}(L)$, $x\mapsto(\mathrm{ad}\alpha(x))|_{L_{1}}$.
If $x\in\mathrm{ker}(\phi)$, then $\mathrm{ad}(\alpha(x))(L_{1})=0$.
Since $\mathrm{ker}(\mathrm{ad}\alpha(x))$ is a $p$-subalgebra of
$G$, we have $x\in C(G)$. It follows that $\mathrm{ker}(\phi)=C(G)$.
Thereby,
\[
\mathrm{dim}_{\mathbb{F}}(G/C(G))=\mathrm{dim}{}_{\mathbb{F}}(G/\mathrm{ker}(\phi))<\mathrm{dim}{}_{\mathbb{F}}\mathrm{ad}(L)<\infty.
\]

(2) Choose a subspace $V\subset C(\hat{L})$ such that $C(\hat{L})=V\oplus(C(\hat{L})\cap L)$.
Be analogous to the case of Lie algebra in \cite{SF}, it can be proved
that the $p$-subalgebra generated by $L$ in $\hat{L}/V$ is the
desired $p$-envelope of $L$.

(3) Considering $L\rightarrow G\hookrightarrow\hat{G}$, the universal
property of $\hat{L}$ ensures the existence of $\hat{f}$

\begin{align*}
L & \begin{array}{ccc}
 & \hookrightarrow\end{array}\hat{L}\\
\begin{array}{cc}
f & \downarrow\end{array} & \begin{array}{ccccc}
 &  &  & \downarrow & \hat{f}\end{array}\\
G & \begin{array}{ccc}
 & \hookrightarrow\end{array}\hat{G}
\end{align*}

If $f$ is onto, then $\hat{f}(\hat{L})\supset f(L)_{p}=G_{p}=\hat{G}$.
Along with $\hat{f}(\hat{L})\subset\hat{G}$ yields that $\hat{f}(\hat{L})=\hat{G}$,
which implies that $\hat{f}$ is surjective.
\end{proof}
Next, these results will be applied to representation theory. To this
end, several definitions and lemmas will be introduced first.

\begin{definition} Let $(L,[\cdot,\cdot]_{L},\alpha)$ be a multiplicative
Hom-Lie algebra.

(1) (see \cite{She}) A representation of $(L,[\cdot,\cdot]_{L},\alpha)$
on the vector space $V$ with respect to $\beta\in\mathrm{gl}(V)$
is a linear map $\rho_{\beta}:L\rightarrow\mathrm{gl}(V)$, such that
for any $x,y\in L$, the following equalities are satisfied:
\[
\rho_{\beta}(\alpha(x))\lyxmathsym{\textopenbullet}\beta=\beta\lyxmathsym{\textopenbullet}\rho_{\beta}(x),\rho_{\beta}([x,y]_{L})\lyxmathsym{\textopenbullet}\beta=\rho_{\beta}(\alpha(x))\rho_{\beta}(y)-\rho_{\beta}(\alpha(y))\rho_{\beta}(x).
\]

(2) The representation $\rho_{\beta}$ is called as faithful if $\textrm{ker}(\rho_{\beta})=0$.

(3) For a restricted Hom-Lie algebra $(L,[\cdot,\cdot]_{L},\alpha,[p]_{\alpha})$,
the representation $\rho_{\beta}$ is called as a $p$-representation
if $\rho_{\beta}(x^{[p]_{\alpha}})=\rho_{\beta}(x)^{p}$, for all
$x\in L$.

(4) For $x\in L$, $\rho_{\beta}(x)$ is said to be nilpotent if there
is $k\in\mathbb{N}$ such that $\rho_{\beta}(x)^{k}=0$.

\end{definition}

Be analogue to Proposition 3.1 in \cite{BM2}, we have

\begin{lemma} Let $(L,[\cdot,\cdot]_{L})$ be a Lie algebra and $\rho:L\rightarrow\mathrm{gl}(V)$
a representation. Let $\alpha:L\rightarrow L$ be a Lie algebra morphism
and $\beta\in\mathrm{gl}(V)$ be a linear map such that $\rho(\alpha(x))\lyxmathsym{\textopenbullet}\beta=\beta\lyxmathsym{\textopenbullet}\rho(x)$,
$\forall x\in L$. Then $(L,[\cdot,\cdot]_{\alpha},\alpha)$ with
$[\cdot,\cdot]_{\alpha}=\alpha\lyxmathsym{\textopenbullet}[\cdot,\cdot]_{L}$
is a Hom-Lie algebra and $\rho_{\beta}:L\rightarrow\mathrm{gl}(V)$,
where $\rho_{\beta}=\beta\circ\rho$, is a representation.

\end{lemma}

The following lemma is Iwasawa's theorem in the Lie algebra case.

\begin{lemma} (see \cite{SF}) Let $L$ be a finite-dimensional Lie
algebra, $G$ a finite-dimensional $p$-envelope of $L$ and $u(G)$
a restricted universal enveloping algebra. Then the mapping $\mathscr{\rho}_{\mathfrak{L}}:L\rightarrow\mathrm{gl}(u(G))$
given by the left multiplication in $u(G)$ is a representation satisfying
the following equivalence relation:
\[
\rho_{\mathfrak{L}}(x)\begin{array}{cc}
nilpotent\end{array}\Longleftrightarrow\begin{array}{cc}
 & \mathrm{ad}x\end{array}nilpotent
\]
holds for all $x\in L$. \end{lemma}

Thus, we obtain the following theorem, which can be seen as the Hom-version
of Iwasawa's theorem.

\begin{theorem}

Let $(L,[\cdot,\cdot]_{\alpha},\alpha)$ be a finite-dimensional involutive
Hom-Lie algebra, where $[\cdot,\cdot]_{\alpha}=\alpha\circ[\cdot,\cdot]$,
$[\cdot,\cdot]$ is a Lie bracket and $\alpha$ is a Lie algebra morphism.
Then $L$ possesses a finite-dimensional faithful representation $\rho_{\beta}$
such that $\mathrm{ad}(x)$ is nilpotent if and only if $\rho_{\beta}(x)$
is nilpotent for all $x\in L_{1}=\{x\in L|\alpha(x)=x\}$.

\end{theorem}
\begin{proof}
Assume that $u(G)$ and $\mathscr{\rho}_{\mathfrak{L}}$ are given
as in Lemma 4.7. Let $\beta\in\mathrm{gl}(u(G))$ satisfying $\rho_{\mathfrak{L}}(\alpha(x))\text{\textopenbullet}\beta=\beta\text{\textopenbullet}\rho_{\mathfrak{L}}(x)$.
Then according to Lemma 4.6, $\rho_{\beta}=\beta\circ\rho_{\mathfrak{L}}:L\rightarrow\mathrm{gl}(u(G))$
is a representation of the Hom-Lie algebra $(L,[\cdot,\cdot]_{\alpha},\alpha)$.
It's not difficult to prove that $\rho_{\beta}$ is the representation
satisfying the requirement.
\end{proof}
In the sequel, we propose to describe all $p$-envelopes of a finite-dimensional
Hom-Lie algebra $(L,[.,.]_{L},\alpha)$ by one of minimal dimension.
A $p$-envelope of $L$ is referred to be minimal if its dimension
is minimal among all $p$-envelopes of $L$. First, we have

\begin{lemma} Let $(L,[.,.]_{L},\alpha)$ be a Hom-Lie algebra and
$(G,[.,.]_{G},\alpha_{G},[p],i)$ a universal $p$-envelope of $L$.
Then

(1) Any vector space $V$ with $\alpha(V)\subset V$ and $V\supset i(L)$
is a Hom-Lie ideal of $G$.

(2) If $V\vartriangleleft G$ is a Hom-Lie ideal with $V\cap i(L)=0$,
then $V\subset\mathfrak{z}(G)$, where $\mathfrak{z}(G)=\{x\in G|[x,y]=0,\:\forall y\in G\}$.

(3) Let $V\vartriangleleft G$ be a Hom-Lie ideal with $V\cap i(L)=0$,
then there exists a $p$-envelope $(G',[.,.]_{G'},\alpha_{G'},[p]',i')$
such that $G'\subset G/V$.

\end{lemma}
\begin{proof}
(1) Note that $[V,G]_{G}\subset[G,G]_{G}\subset i(L)\subset V$. Thereby,
$[V,G]_{G}\subset V$. Along with $\alpha(V)\subset V$, one gets
that $V$ is a Hom-Lie ideal of $G$.

(2) Since $V$ is a Hom-Lie ideal of $G$, we have $[V,G]_{G}\subset V$.
It's known from (1) that $[V,G]_{G}\subset i(L)$, and the assumption
$V\cap i(L)=0$ implies that $[V,G]_{G}=0$. Therefore, $V\subset\mathfrak{z}(G)$.

(3) Let $\pi:G\rightarrow G/V$ be the canonical projection. By virtue
of Theorem 4.3, the algebra $G/V$ is restrictable. Let $[p]':G/V\rightarrow G/V$
be a $p$-structure on $G/V$ and $G'$ the $p$-subalgebra of $G/V$
generated by $\pi\circ i(L)$. Clearly, $i':=\pi\circ i$ is injective.
Therefore, the 5-tuple $(G',[.,.]_{G},\alpha_{G'},[p]',i')$ is the
desired $p$-envelope of $L$.
\end{proof}
Based on this lemma, we obtain the following proposition, which proof
is analogous to the case of Lie algebras (see \cite{SF}):

\begin{proposition} Let $(G,[.,.]_{G},\alpha_{G},[p],i)$ and $(G',[.,.]_{G'},\alpha_{G'},[p]',i')$
be two $p$-envelopes of the Hom-Lie algebra $(L,[.,.]_{L},\alpha)$,
then there exists a (nonrestricted) morphism $f:G\rightarrow G'$
such that $f\circ i=i'$.

\end{proposition}

\begin{proposition} Let $(L,[.,.]_{L},\alpha)$ be a finite-dimensional
involutive Hom-Lie algebra, $(G,[.,.]_{G},$ $\alpha_{G},[p],i)$
and $(G',[.,.]_{G'},\alpha_{G'},[p]',i')$ two finite-dimensional
$p$-envelopes of $L$. Assume that $f:G\rightarrow G'$ is a morphism
such that $f\circ i=i'$. Then the following statements hold.

(1) There is a Hom-Lie ideal $J\subset C(G')$ such that $G'=f(G)\oplus J$.

(2) There exists a $p$-envelope $H$ of $L$ with $H\subset f(G)$.

\end{proposition}
\begin{proof}
(1) According to Proposition 4.9, there exists a morphism $g:G'\rightarrow G$
with $g\circ i'=i$. Set $\mu:=f\circ g$ and consider the Fitting
decomposition of $G'=G'_{0}\oplus G'_{1}$ with respect to $\mu$.
Since $\mu\circ i'=f\circ g\circ i'=f\circ i=i'$, one gets $\mu^{k}\circ i'=i'$,
for any $k\in\mathbb{N}$. Without loss of generality, suppose that
$G'_{0}=\mathrm{ker}(\mu^{n})$. Then $G'_{0}$ is a Hom-Lie ideal
of $G'$. It follows from $\mu^{n}\circ i'=i'$ that $G'_{0}\cap i'(L)=0$.
And by virtue of Lemma 4.8, one gets $G'_{0}\subset C(G')$. Along
with $G'_{1}=\mu(G'_{1})\subset f(G)$, we have $G'=C(G')+f(G)$.
Furthermore, a direct complement $J$ of $f(G)$ with $J\subset C(G')$
will be chosen such that $G'=f(G)\oplus J$.

(2) Since $G$ is restrictable, $f(G)$ is restrictable. Thus we may
choose a $p$-structure $[p]''$ on $f(G)$ and let $H\subset f(G)$
be the $p$-subalgebra generated by $i'(L)$. Therefore, the 5-tuple
$(H,[.,.]_{G'},(\alpha_{G'})|_{H},[p]'',i')$ is the desired $p$-envelope
of $L$.
\end{proof}
\begin{theorem}Let $(L,[.,.]_{L},\alpha_{L})$ be a finite-dimensional
involutive Hom-Lie algebra. Then the following statements hold.

(1) Any two minimal $p$-envelopes of $L$ are isomorphic as ordinary
Hom-Lie algebras.

(2) If $(G,[.,.]_{G},\alpha_{G},[p],i)$ is a finite-dimensional $p$-envelope
of $L$, then there is a minimal $p$-envelope $H\subset G$ and a
Hom-Lie ideal $J\subset C(G)$ such that $G=H\oplus J$ and $i(L)\subset H$.

(3) A finite-dimensional $p$-envelope $(G,[.,.]_{G},\alpha_{G},[p],i)$
is minimal if and only if $C(G)$ $\subset i(L)$.

\end{theorem}
\begin{proof}
(1) Let $(G,[.,.]_{G},\alpha_{G},[p],i)$ and $(G',[.,.]_{G'},\alpha_{G'},[p]',i')$
be two minimal $p$-envelopes of $L$. According to Proposition 4.9,
there exists a morphism $f:G\rightarrow G'$ such that $f\circ i=i'$.
And Proposition 4.10 yields the existence of a Hom-Lie ideal $J\subset C(G')$
and a $p$-envelope $H$ such that $G'=f(G)\oplus J$ and $H\subset f(G)$.
Moreover, it follows from the minimality of $G'$ that $H=G'$. Thus
$f(G)=G'$ and $f$ is surjective. Along with $\mathrm{dim}\mathbb{_{F}}G=\mathrm{dim}\mathbb{_{F}}G'$,
$f$ is bijective.

(2) Let $(G',[.,.]_{G'},\alpha_{G'},[p]',i')$ be a minimal $p$-envelope
of $L$. According to (1) in Proposition 4.10, we may decompose $G=g(G')\oplus J$,
where $g:G'\rightarrow G$ is a morphism such that $g\circ i'=i$.
Moreover, by virtue of (2) in Proposition 4.10, a $p$-envelope $H$
of $L$ with $H\subset g(G')$ can be found. And the minimality of
$G'$ entails that $H=g(G')$. Therefore, $G=H\oplus J$.

(3) ($\Rightarrow$) Suppose that $G$ is minimal. Let $I\triangleleft G$
be a Hom-Lie ideal with $I\cap i(L)=0$. We may write $C(G)=C(G)\cap i(L)\oplus I$.
According to (3) in Lemma 4.8, there exists a $p$-envelope $G'$
with $G'\subset G/I$. And the minimality of $G$ entails that $\mathrm{dim}{}_{\mathbb{F}}G<\mathrm{dim}{}_{\mathbb{F}}G'<\mathrm{dim}{}_{\mathbb{F}}G/I$.
Thus $I=0$ and $C(G)=C(G)\cap i(L)$, that is, $C(G)\subset i(L)$.

($\Leftarrow$) Conversely, assume that $C(G)\subset i(L)$. By (2),
we may write $G=H\oplus J$, where $H$ is a minimal $p$-envelope
of $L$ with $i(L)\subset H$ and $J\subset C(G)$ is a Hom-Lie ideal.
The assumption $C(G)\subset i(L)$ implies that $J\subset i(L)\subset H$,
that is, $J=0$. Thus $G=H$ is minimal.
\end{proof}

\end{document}